\documentclass[12pt, reqno]{amsart}
\usepackage{amsmath, amsthm, amscd, amsfonts, amssymb, graphicx, color}
\usepackage[bookmarksnumbered, colorlinks, plainpages]{hyperref}

\newtheorem{theorem}{Theorem}[section]

\theoremstyle{definition}
\newtheorem{definition}[theorem]{Definition}

\theoremstyle{remark}

\numberwithin{equation}{section}

\begin{document}
	
	\setcounter{page}{1}

	\title[Mersenne-Lerch Interpolation Values]{Mersenne-Lerch Interpolation Values and
		Apostol-Mersenne Polynomial Families}
	
	\author[N. Lacpao]{Noel B. Lacpao}
	
	\address{Department of Mathematics, Bukidnon State University, Malaybalay City, Philippines.}
	\email{\textcolor[rgb]{0.00,0.00,0.84}{noel.lacpao@buksu.edu.ph}}
	
\begin{abstract}
	This paper introduces Bernoulli-type and Euler-type Mersenne-Lerch
	interpolation families associated with Apostol-Mersenne polynomial families.
	Their construction is based on the Mersenne translation polynomials
	\(P_{n,M}(x;m)\), defined by
	\[
	e_M^{xt}(e_M^t)^m
	=
	\sum_{n=0}^{\infty}P_{n,M}(x;m)\frac{t^n}{M_n!}.
	\]
	Explicit formulas for these polynomials are derived, including an expansion
	in terms of M-Stirling numbers of the second kind. At nonpositive integers,
	the resulting interpolation values recover the
	Apostol-Mersenne-Bernoulli and Apostol-Mersenne-Euler polynomials of
	order \(r\). Derivative, integral, addition, and difference formulas are
	also obtained for these values. A comparison with \(q\)-calculus shows that the M-factorial, M-binomial
	coefficient, M-derivative, and M-exponential are obtained by setting
	\(q=2\) in the corresponding \(q\)-calculus expressions.
	\medskip
	
	
	\smallskip
	
	
\end{abstract}
	
	\maketitle
	
	\section{Introduction}
	
	Bernoulli and Euler polynomials occur in finite differences, interpolation
	formulas, special values of zeta functions, and series expansions. They are
	usually introduced by the generating functions
	\[
	\frac{t e^{xt}}{e^t-1}
	=
	\sum_{n=0}^{\infty}B_n(x)\frac{t^n}{n!},
	\qquad
	\frac{2e^{xt}}{e^t+1}
	=
	\sum_{n=0}^{\infty}E_n(x)\frac{t^n}{n!}.
	\]
	Apostol introduced a parameter \(\lambda\) into the Bernoulli generating
	function in his study of the Lerch zeta function \cite{Apostol1951}. This led
	to the Apostol-Bernoulli polynomials and later to Apostol-type Euler
	polynomials and their higher-order forms. Luo and Srivastava studied these
	families and connected them with the Hurwitz-Lerch zeta function
	\cite{LuoSrivastava2005,LuoSrivastava2006}.
	
	A different construction replaces the ordinary factorials and derivatives
	with analogues based on the Mersenne numbers
	\[
	M_n=2^n-1.
	\]
	Suna and Ray studied the Mersenne-Bernoulli and Mersenne-Euler polynomials
	and derived addition, difference, derivative, integral, convolution, and
	matrix identities for these families \cite{SunaRay2026Mersenne}. In a
	separate work, they introduced higher-order Apostol-Mersenne-Bernoulli and
	Apostol-Mersenne-Euler polynomials \cite{SunaRay2026}. Their constructions
	use the M-factorial \(M_n!\), the M-binomial coefficient
	\begin{equation}\label{Mbinomialcoeffi}
		\binom{n}{k}_{M}
		=
		\frac{M_n!}{M_k!M_{n-k}!},
	\end{equation}
	and the M-exponential function
	\begin{equation}\label{Mexponentialcoeffi}
		e_M^{xt}
		=
		\sum_{n=0}^{\infty}\frac{x^nt^n}{M_n!}.
	\end{equation}
	The present construction recovers the higher-order Apostol-Mersenne
	polynomial families from interpolation values at nonpositive integers.
	
	Related polynomial families have also been studied in the setting of
	poly-Bernoulli, poly-Euler, and poly-Cauchy numbers. Hurwitz-Lerch-type
	multi-poly-Cauchy numbers were defined through a multiple polylogarithm
	factorial function in \cite{LacpaoCorcinoVega2019}. Generalized
	multi-poly-Euler and multi-poly-Bernoulli polynomials were studied in relation
	to Hurwitz-Lerch multiple zeta values in
	\cite{CorcinoJolanyCorcinoKomatsu2017}. Interpolation formulas for generalized
	poly-Bernoulli polynomials and zeta-type functions were obtained in
	\cite{JolanyCorcino2015}. The same interpolation viewpoint is used here for Apostol-Mersenne
	polynomial families.
	
	The construction begins with the Mersenne translation polynomials
	\(P_{n,M}(x;m)\), defined by
	\[
	e_M^{xt}\left(e_M^t\right)^m
	=
	\sum_{n=0}^{\infty}
	P_{n,M}(x;m)\frac{t^n}{M_n!}.
	\]
	These polynomials play the role of the ordinary powers \((x+m)^n\) in the
	Mersenne setting. They are defined through products of M-exponential
	functions and avoid assigning a meaning to the formal expression
	\((x+_M m)^s\) when \(s\) is not a nonnegative integer.
	
	For \(r\in\mathbb{N}\), the Bernoulli-type and Euler-type Mersenne-Lerch
	interpolation values are defined by
	\[
	\Phi^{B,r}_{M,\lambda}(-n,x)
	=
	\sum_{m=0}^{\infty}
	\binom{m+r-1}{m}
	\lambda^mP_{n,M}(x;m),
	\]
	and
	\[
	\Phi^{E,r}_{M,\lambda}(-n,x)
	=
	\sum_{m=0}^{\infty}
	(-1)^m
	\binom{m+r-1}{m}
	\lambda^mP_{n,M}(x;m),
	\]
	where \(|\lambda|<1\). Their generating functions are
	\[
	\sum_{n=0}^{\infty}
	\Phi^{B,r}_{M,\lambda}(-n,x)\frac{t^n}{M_n!}
	=
	\frac{e_M^{xt}}{(1-\lambda e_M^t)^r},
	\]
	and
	\[
	\sum_{n=0}^{\infty}
	\Phi^{E,r}_{M,\lambda}(-n,x)\frac{t^n}{M_n!}
	=
	\frac{e_M^{xt}}{(1+\lambda e_M^t)^r}.
	\]
	These generating functions give
	\[
	\Phi^{B,r}_{M,\lambda}(-n,x)
	=
	(-1)^r
	\frac{M_n!}{M_{n+r}!}
	B^{(r)}_{n+r,M}(x;\lambda),
	\]
	and
	\[
	\Phi^{E,r}_{M,\lambda}(-n,x)
	=
	\frac{1}{2^r}
	E^{(r)}_{n,M}(x;\lambda).
	\]
	
	The paper develops the Mersenne translation polynomials, derives their
	M-multinomial and M-Stirling representations, and connects their weighted
	sums with higher-order Apostol-Mersenne polynomial families. The derivative, integral, addition,
	and difference formulas follow from the resulting generating functions.
	
	\section{Preliminaries}
	
	We use the following notation throughout the paper. Let
	\[
	M_n=2^n-1,\qquad n\geq 0,
	\]
	be the \(n\)-th Mersenne number. Following the M-calculus notation for
	Apostol-Mersenne polynomials in \cite{SunaRay2026}, define
	\[
	M_0!=1,
	\qquad
	M_n!=M_nM_{n-1}\cdots M_1
	\quad (n\geq 1).
	\]
	For \(0\leq k\leq n\), the M-binomial coefficient is
	\[
	\binom{n}{k}_{M}
	=
	\frac{M_n!}{M_k!M_{n-k}!}.
	\]
	We also use the convention
	\[
	\binom{n}{k}_{M}=0
	\qquad (k>n).
	\]
	
	The M-exponential function is defined by
	\[
	e_M^{xt}
	=
	\sum_{n=0}^{\infty}
	\frac{x^nt^n}{M_n!}.
	\]
	The M-binomial addition \(x+_M y\) is determined by
	\[
	(x+_M y)^n
	=
	\sum_{k=0}^{n}
	\binom{n}{k}_{M}
	x^ky^{n-k}.
	\]
	It follows that
	\[
	e_M^{xt}e_M^{yt}
	=
	e_M^{(x+_M y)t}.
	\]
	
	For a function \(f\), the M-derivative is given by
	\[
	D_xf(x)
	=
	\frac{f(2x)-f(x)}{x},
	\qquad x\neq 0,
	\]
	with
	\[
	D_xf(0)=\lim_{x\to 0}D_xf(x)
	\]
	whenever the limit exists. In particular,
	\[
	D_x(x^n)=M_nx^{n-1},
	\qquad n\geq 1.
	\]
	
	The M-integral is taken as the inverse operation of the M-derivative. If
	\(D_xF(x)=f(x)\), then
	\begin{equation}\label{Mintegral}
		\int_a^b f(x)\,d_Mx
		=
		F(b)-F(a).
	\end{equation}
	For monomials, this gives
	\[
	\int_0^1 x^n\,d_Mx
	=
	\frac{1}{M_{n+1}},
	\qquad n\geq 0.
	\]
	
	We also use the Apostol-Mersenne-Bernoulli and
	Apostol-Mersenne-Euler polynomials of order \(r\). For
	\(r\in\mathbb{N}\), their generating functions are
	\cite{SunaRay2026}
	\begin{equation}\label{AposMerBern}
		\left(\frac{t}{\lambda e_M^t-1}\right)^r e_M^{xt}
		=
		\sum_{n=0}^{\infty}
		B^{(r)}_{n,M}(x;\lambda)\frac{t^n}{M_n!},
	\end{equation}
	and
	\begin{equation}\label{AposMerEu}
		\left(\frac{2}{\lambda e_M^t+1}\right)^r e_M^{xt}
		=
		\sum_{n=0}^{\infty}
		E^{(r)}_{n,M}(x;\lambda)\frac{t^n}{M_n!}.
	\end{equation}
	
	The choices \(r=1\) and \(\lambda=1\) in
	\eqref{AposMerBern} and \eqref{AposMerEu} give the generating functions
	of the Mersenne-Bernoulli and Mersenne-Euler polynomials studied in
	\cite{SunaRay2026Mersenne}. This specialization concerns the polynomial
	generating functions only. The interpolation values introduced below are
	defined by convergent series under the condition \(|\lambda|<1\), so the
	case \(\lambda=1\) is not included in their stated domain. Thus, the basic
	Mersenne polynomial families are not obtained by substituting
	\(\lambda=1\) directly into the defining interpolation series.
	
	When \(x=0\), we write
	\[
	B^{(r)}_{n,M}(\lambda)
	=
	B^{(r)}_{n,M}(0;\lambda),
	\qquad
	E^{(r)}_{n,M}(\lambda)
	=
	E^{(r)}_{n,M}(0;\lambda).
	\]
	
	The Apostol-type Bernoulli and Euler polynomials go back to Apostol's work on
	the Lerch zeta function \cite{Apostol1951}. Luo and Srivastava later studied
	their higher-order forms and their relation with the Hurwitz-Lerch zeta
	function \cite{LuoSrivastava2005,LuoSrivastava2006}. The same
	generating-function approach is used here with the M-factorial \(M_n!\) and
	the M-exponential \(e_M^{xt}\).
	
	We shall also use the falling factorial
	\[
	(m)_0=1,
	\qquad
	(m)_j=m(m-1)\cdots(m-j+1)
	\quad (j\geq 1).
	\]
	For \(r\in\mathbb{N}\), the negative binomial expansion is
	\[
	\frac{1}{(1-z)^r}
	=
	\sum_{m=0}^{\infty}
	\binom{m+r-1}{m}z^m,
	\qquad |z|<1.
	\]
	
	The Hurwitz-Lerch factorial zeta function appears in the study of
	poly-Cauchy and poly-Bernoulli-type numbers and polynomials. One form is
	\[
	\Phi_f(z,s,a)
	=
	\sum_{n=0}^{\infty}
	\frac{z^n}{n!(n+a)^s},
	\qquad
	a\notin\{0,-1,-2,\ldots\},
	\]
	with the usual restrictions on \(z\) and \(s\). It has been used to define
	Hurwitz-Lerch-type poly-Cauchy and poly-Bernoulli objects
	\cite{CenkciYoung2015,Lacpao2023}. In the present construction, the
	interpolation values are formed from the Mersenne translation polynomials
	and the factorial \(M_n!\).
	
	\section{Mersenne translation polynomials}
	
	The interpolation construction developed below is based on a family of Mersenne translation polynomials. Let \(m\in\mathbb{N}_0\) and
\(x\in\mathbb{C}\).
	
	\begin{definition}\label{MersenneTransPol}
		For \(m\in\mathbb{N}_0\), the Mersenne translation polynomials
		\(P_{n,M}(x;m)\) are defined by
		\[
		e_M^{xt}\left(e_M^t\right)^m
		=
		\sum_{n=0}^{\infty}
		P_{n,M}(x;m)\frac{t^n}{M_n!}.
		\]
		In coefficient form,
		\[
		P_{n,M}(x;m)
		=
		M_n!\,[t^n]
		\left(
		e_M^{xt}\left(e_M^t\right)^m
		\right),
		\qquad n\geq0.
		\]
	\end{definition}
	
	The factor \(\left(e_M^t\right)^m\) is the ordinary \(m\)-fold product of
	\(e_M^t\), so \(m\) is taken to be a nonnegative integer. The polynomial
	\(P_{n,M}(x;m)\) plays the role of \((x+m)^n\) in the Mersenne setting, but
	its definition uses products of M-exponential functions rather than ordinary
	addition.
	
	\begin{theorem}
		Let \(m\in\mathbb{N}_0\) and \(n\geq0\). Then
		\[
		P_{n,M}(x;m)
		=
		\sum_{\substack{
				r_0,r_1,\ldots,r_m\geq0\\
				r_0+r_1+\cdots+r_m=n
		}}
		\frac{M_n!}
		{M_{r_0}!M_{r_1}!\cdots M_{r_m}!}
		x^{r_0}.
		\]
		If the M-multinomial coefficient is defined by
		\[
		\binom{n}{r_0,r_1,\ldots,r_m}_M
		=
		\frac{M_n!}
		{M_{r_0}!M_{r_1}!\cdots M_{r_m}!},
		\qquad
		r_0+r_1+\cdots+r_m=n,
		\]
		then
		\[
		P_{n,M}(x;m)
		=
		\sum_{\substack{
				r_0,r_1,\ldots,r_m\geq0\\
				r_0+r_1+\cdots+r_m=n
		}}
		\binom{n}{r_0,r_1,\ldots,r_m}_M
		x^{r_0}.
		\]
	\end{theorem}
	
	\begin{proof}
		By Definition~\ref{MersenneTransPol},
		\[
		\sum_{n=0}^{\infty}
		P_{n,M}(x;m)\frac{t^n}{M_n!}
		=
		e_M^{xt}\left(e_M^t\right)^m.
		\]
		Expanding the right-hand side gives
		\begin{align*}
			e_M^{xt}\left(e_M^t\right)^m
			&=
			\left(
			\sum_{r_0=0}^{\infty}
			\frac{x^{r_0}t^{r_0}}{M_{r_0}!}
			\right)
			\prod_{j=1}^{m}
			\left(
			\sum_{r_j=0}^{\infty}
			\frac{t^{r_j}}{M_{r_j}!}
			\right)
			\\
			&=
			\sum_{r_0,r_1,\ldots,r_m\geq0}
			\frac{
				x^{r_0}t^{r_0+r_1+\cdots+r_m}
			}{
				M_{r_0}!M_{r_1}!\cdots M_{r_m}!
			}
			\\
			&=
			\sum_{n=0}^{\infty}
			\left(
			\sum_{\substack{
					r_0,r_1,\ldots,r_m\geq0\\
					r_0+r_1+\cdots+r_m=n
			}}
			\frac{x^{r_0}}
			{M_{r_0}!M_{r_1}!\cdots M_{r_m}!}
			\right)t^n
			\\
			&=
			\sum_{n=0}^{\infty}
			\left(
			\sum_{\substack{
					r_0,r_1,\ldots,r_m\geq0\\
					r_0+r_1+\cdots+r_m=n
			}}
			\frac{
				M_n!x^{r_0}
			}{
				M_{r_0}!M_{r_1}!\cdots M_{r_m}!
			}
			\right)
			\frac{t^n}{M_n!}.
		\end{align*}
		Comparing coefficients gives the result.
	\end{proof}
	
	When \(m=0\), the defining relation reduces to
	\[
	P_{n,M}(x;0)=x^n.
	\]
	For \(m=1\),
	\[
	P_{n,M}(x;1)
	=
	\sum_{k=0}^{n}
	\binom{n}{k}_M x^k
	=
	(x+_M1)^n.
	\]
	Thus, \(P_{n,M}(x;m)\) extends the M-binomial translation from one
	M-exponential factor to \(m\) repeated factors.
	
	The next result describes the action of the M-derivative on these
	polynomials.
	
	\begin{theorem}
		Let \(m\in\mathbb{N}_0\). For \(n\geq1\),
		\[
		D_xP_{n,M}(x;m)
		=
		M_nP_{n-1,M}(x;m).
		\]
		When \(n=0\),
		\[
		D_xP_{0,M}(x;m)=0.
		\]
	\end{theorem}
	
	\begin{proof}
		From the definition of the M-exponential function,
		\begin{align*}
			D_xe_M^{xt}
			&=
			\sum_{n=1}^{\infty}
			M_nx^{n-1}\frac{t^n}{M_n!}
			\\
			&=
			t\sum_{n=0}^{\infty}
			\frac{x^nt^n}{M_n!}
			\\
			&=
			te_M^{xt}.
		\end{align*}
		Applying \(D_x\) to the generating function in
		Definition~\ref{MersenneTransPol}, we obtain
		\begin{align*}
			\sum_{n=0}^{\infty}
			D_xP_{n,M}(x;m)\frac{t^n}{M_n!}
			&=
			D_x\left(
			e_M^{xt}\left(e_M^t\right)^m
			\right)
			\\
			&=
			te_M^{xt}\left(e_M^t\right)^m
			\\
			&=
			t\sum_{n=0}^{\infty}
			P_{n,M}(x;m)\frac{t^n}{M_n!}
			\\
			&=
			\sum_{n=1}^{\infty}
			M_nP_{n-1,M}(x;m)
			\frac{t^n}{M_n!}.
		\end{align*}
		For \(n\geq1\), comparing the coefficients of
		\(\displaystyle t^n/M_n!\) gives
		\[
		D_xP_{n,M}(x;m)
		=
		M_nP_{n-1,M}(x;m).
		\]
		The right-hand side has no constant term, so
		\[
		D_xP_{0,M}(x;m)=0.
		\]
	\end{proof}
	
	The derivative formula shows that \(P_{n,M}(x;m)\) satisfies the same
	lowering relation as \(x^n\), with \(M_n\) in place of the ordinary factor
	\(n\).
	
	\section{M-Stirling numbers and explicit expansions}
	
	The M-Stirling numbers of the second kind provide a finite-sum expression for
	the Mersenne translation polynomials in terms of the integer parameter \(m\).
	
	\begin{definition}
		For \(j\in\mathbb{N}_0\), the M-Stirling numbers of the second kind,
		denoted by \(S_M(n,j)\), are defined by
		\[
		\frac{(e_M^t-1)^j}{j!}
		=
		\sum_{n=0}^{\infty}
		S_M(n,j)\frac{t^n}{M_n!}.
		\]
		In coefficient form,
		\[
		S_M(n,j)
		=
		\frac{M_n!}{j!}\,[t^n](e_M^t-1)^j,
		\qquad n,j\geq0.
		\]
	\end{definition}
	
	It follows from the definition that
	\[
	S_M(0,0)=1,
	\qquad
	S_M(n,0)=0
	\quad (n\geq1).
	\]
	Since \((e_M^t-1)^j\) has no term of degree less than \(j\),
	\[
	S_M(n,j)=0
	\qquad (j>n).
	\]
	
	\begin{theorem}\label{exp1}
		Let \(m\in\mathbb{N}_0\) and \(n\geq0\). Then
		\[
		P_{n,M}(x;m)
		=
		\sum_{\ell=0}^{n}
		\binom{n}{\ell}_M
		x^\ell
		\sum_{j=0}^{n-\ell}
		(m)_jS_M(n-\ell,j),
		\]
		where
		\[
		(m)_0=1,
		\qquad
		(m)_j=m(m-1)\cdots(m-j+1)
		\quad (j\geq1).
		\]
	\end{theorem}
	
	\begin{proof}
		By Definition~\ref{MersenneTransPol},
		\[
		\sum_{n=0}^{\infty}
		P_{n,M}(x;m)\frac{t^n}{M_n!}
		=
		e_M^{xt}\left(e_M^t\right)^m.
		\]
		Using
		\[
		\left(e_M^t\right)^m
		=
		\left(1+(e_M^t-1)\right)^m,
		\]
		we obtain
		\begin{align*}
			e_M^{xt}\left(e_M^t\right)^m
			&=
			\left(
			\sum_{\ell=0}^{\infty}
			\frac{x^\ell t^\ell}{M_\ell!}
			\right)
			\left(
			\sum_{j=0}^{m}
			\binom{m}{j}(e_M^t-1)^j
			\right)
			\\
			&=
			\left(
			\sum_{\ell=0}^{\infty}
			\frac{x^\ell t^\ell}{M_\ell!}
			\right)
			\left(
			\sum_{j=0}^{m}
			(m)_j\frac{(e_M^t-1)^j}{j!}
			\right)
			\\
			&=
			\left(
			\sum_{\ell=0}^{\infty}
			\frac{x^\ell t^\ell}{M_\ell!}
			\right)
			\left(
			\sum_{j=0}^{m}
			(m)_j
			\sum_{q=0}^{\infty}
			S_M(q,j)\frac{t^q}{M_q!}
			\right).
		\end{align*}
		Since \(S_M(q,j)=0\) for \(j>q\) and \((m)_j=0\) for \(j>m\),
		the second factor can be written as
		\[
		\sum_{q=0}^{\infty}
		\left(
		\sum_{j=0}^{q}
		(m)_jS_M(q,j)
		\right)
		\frac{t^q}{M_q!}.
		\]
		Therefore,
		\begin{align*}
			e_M^{xt}\left(e_M^t\right)^m
			&=
			\sum_{n=0}^{\infty}
			\left[
			\sum_{\ell=0}^{n}
			\frac{x^\ell}
			{M_\ell!M_{n-\ell}!}
			\sum_{j=0}^{n-\ell}
			(m)_jS_M(n-\ell,j)
			\right]t^n
			\\
			&=
			\sum_{n=0}^{\infty}
			\left[
			\sum_{\ell=0}^{n}
			\binom{n}{\ell}_M
			x^\ell
			\sum_{j=0}^{n-\ell}
			(m)_jS_M(n-\ell,j)
			\right]
			\frac{t^n}{M_n!}.
		\end{align*}
		Comparing the coefficients of \(\displaystyle t^n/M_n!\) gives the result.
	\end{proof}
	
	The formula separates the powers of \(x\) from the dependence on \(m\), which
	appears through the falling factorials \((m)_j\).
	
	\section{Bernoulli-type Mersenne-Lerch interpolation}
	
	The Mersenne translation polynomials give a Bernoulli-type interpolation
	family whose values at nonpositive integers are connected with the
	Apostol-Mersenne-Bernoulli polynomials of order \(r\).
	
	\begin{definition}\label{BernMersenneLerchInt}
		Let \(r\in\mathbb{N}\), \(n\in\mathbb{N}_0\),
		\(x\in\mathbb{C}\), and \(|\lambda|<1\). The Bernoulli-type
		Mersenne-Lerch interpolation value of order \(r\) at the nonpositive
		integer \(-n\) is defined by
		\[
		\Phi^{B,r}_{M,\lambda}(-n,x)
		=
		\sum_{m=0}^{\infty}
		\binom{m+r-1}{m}
		\lambda^mP_{n,M}(x;m).
		\]
	\end{definition}
	
	By Theorem \ref{exp1}, for fixed \(n\), the polynomial \(P_{n,M}(x;m)\) has degree
	at most \(n\) in \(m\). Hence, the defining series converges for \(|\lambda|<1\). When
	\(r=1\), we write
	\[
	\Phi^{B}_{M,\lambda}(-n,x)
	=
	\Phi^{B,1}_{M,\lambda}(-n,x)
	=
	\sum_{m=0}^{\infty}
	\lambda^mP_{n,M}(x;m).
	\]
	
	The generating function follows from the defining relation for
	\(P_{n,M}(x;m)\) and the negative binomial expansion.
	
	\begin{theorem}\label{ApostolMersBern}
		Let \(r\in\mathbb{N}\), \(x\in\mathbb{C}\), and
		\(|\lambda|<1\). Then
		\[
		\sum_{n=0}^{\infty}
		\Phi^{B,r}_{M,\lambda}(-n,x)\frac{t^n}{M_n!}
		=
		\frac{e_M^{xt}}{(1-\lambda e_M^t)^r},
		\]
		for all \(t\) in a neighborhood of \(0\) satisfying
		\[
		|\lambda e_M^t|<1.
		\]
	\end{theorem}
	
	\begin{proof}
		By Definition~\ref{BernMersenneLerchInt},
		\[
		\sum_{n=0}^{\infty}
		\Phi^{B,r}_{M,\lambda}(-n,x)\frac{t^n}{M_n!}
		=
		\sum_{n=0}^{\infty}
		\left(
		\sum_{m=0}^{\infty}
		\binom{m+r-1}{m}
		\lambda^mP_{n,M}(x;m)
		\right)
		\frac{t^n}{M_n!}.
		\]
		
		Since \(|\lambda|<1\), choose \(\delta>0\) such that
		\[
		|\lambda|e_M^\delta<1.
		\]
		For \(|t|\leq\delta\),
		\[
		\sum_{n=0}^{\infty}
		\left|P_{n,M}(x;m)\right|
		\frac{|t|^n}{M_n!}
		\leq
		e_M^{|x||t|}
		\left(e_M^{|t|}\right)^m.
		\]
		Thus,
		\[
		\sum_{m=0}^{\infty}
		\binom{m+r-1}{m}
		|\lambda|^m
		\sum_{n=0}^{\infty}
		\left|P_{n,M}(x;m)\right|
		\frac{|t|^n}{M_n!}
		\]
		is bounded by
		\[
		e_M^{|x||t|}
		\sum_{m=0}^{\infty}
		\binom{m+r-1}{m}
		\left(
		|\lambda|e_M^{|t|}
		\right)^m,
		\]
		which is finite. The order of summation may therefore be changed.
		
		Using Definition~\ref{MersenneTransPol}, we obtain
		\begin{align*}
			\sum_{n=0}^{\infty}
			\Phi^{B,r}_{M,\lambda}(-n,x)\frac{t^n}{M_n!}
			&=
			\sum_{m=0}^{\infty}
			\binom{m+r-1}{m}
			\lambda^m
			\sum_{k=0}^{\infty}
			P_{k,M}(x;m)\frac{t^k}{M_k!}
			\\
			&=
			\sum_{m=0}^{\infty}
			\binom{m+r-1}{m}
			\lambda^m
			e_M^{xt}\left(e_M^t\right)^m
			\\
			&=
			e_M^{xt}
			\sum_{m=0}^{\infty}
			\binom{m+r-1}{m}
			\left(\lambda e_M^t\right)^m.
		\end{align*}
		For \(|\lambda e_M^t|<1\), the negative binomial expansion gives
		\[
		\sum_{m=0}^{\infty}
		\binom{m+r-1}{m}
		\left(\lambda e_M^t\right)^m
		=
		\frac{1}{(1-\lambda e_M^t)^r}.
		\]
		Hence,
		\[
		\sum_{n=0}^{\infty}
		\Phi^{B,r}_{M,\lambda}(-n,x)\frac{t^n}{M_n!}
		=
		\frac{e_M^{xt}}{(1-\lambda e_M^t)^r}.
		\]
	\end{proof}
	
	The factor \(t^r\) and the sign in the generating function of the
	Apostol-Mersenne-Bernoulli polynomials give the following interpolation
	formula.
	
	\begin{theorem}
		Let \(r\in\mathbb{N}\), \(n\in\mathbb{N}_0\),
		\(x\in\mathbb{C}\), and \(|\lambda|<1\). Then
		\[
		\Phi^{B,r}_{M,\lambda}(-n,x)
		=
		(-1)^r
		\frac{M_n!}{M_{n+r}!}
		B^{(r)}_{n+r,M}(x;\lambda),
		\]
		where \(B^{(r)}_{n,M}(x;\lambda)\) denotes the
		Apostol-Mersenne-Bernoulli polynomial of order \(r\).
	\end{theorem}
	
	\begin{proof}
		From the generating function in \eqref{AposMerBern},
		\[
		\left(
		\frac{t}{\lambda e_M^t-1}
		\right)^r
		e_M^{xt}
		=
		\sum_{k=0}^{\infty}
		B^{(r)}_{k,M}(x;\lambda)
		\frac{t^k}{M_k!}.
		\]
		Since
		\[
		1-\lambda e_M^t
		=
		-\left(\lambda e_M^t-1\right),
		\]
		we have
		\[
		\frac{e_M^{xt}}{(1-\lambda e_M^t)^r}
		=
		(-1)^rt^{-r}
		\left(
		\frac{t}{\lambda e_M^t-1}
		\right)^r
		e_M^{xt}.
		\]
		Therefore,
		\[
		\frac{e_M^{xt}}{(1-\lambda e_M^t)^r}
		=
		(-1)^rt^{-r}
		\sum_{k=0}^{\infty}
		B^{(r)}_{k,M}(x;\lambda)
		\frac{t^k}{M_k!}.
		\]
		
		Because \(|\lambda|<1\), we have \(\lambda\neq1\). Hence,
		\[
		\left(
		\frac{t}{\lambda e_M^t-1}
		\right)^r
		e_M^{xt}
		\]
		has a zero of order \(r\) at \(t=0\). It follows that
		\[
		B^{(r)}_{k,M}(x;\lambda)=0,
		\qquad 0\leq k<r.
		\]
		After shifting the index,
		\[
		\frac{e_M^{xt}}{(1-\lambda e_M^t)^r}
		=
		(-1)^r
		\sum_{n=0}^{\infty}
		B^{(r)}_{n+r,M}(x;\lambda)
		\frac{t^n}{M_{n+r}!}.
		\]
		On the other hand, Theorem~\ref{ApostolMersBern} gives
		\[
		\frac{e_M^{xt}}{(1-\lambda e_M^t)^r}
		=
		\sum_{n=0}^{\infty}
		\Phi^{B,r}_{M,\lambda}(-n,x)
		\frac{t^n}{M_n!}.
		\]
		Comparing the coefficients of \(t^n\), we obtain
		\[
		\frac{\Phi^{B,r}_{M,\lambda}(-n,x)}{M_n!}
		=
		(-1)^r
		\frac{B^{(r)}_{n+r,M}(x;\lambda)}{M_{n+r}!}.
		\]
		Thus,
		\[
		\Phi^{B,r}_{M,\lambda}(-n,x)
		=
		(-1)^r
		\frac{M_n!}{M_{n+r}!}
		B^{(r)}_{n+r,M}(x;\lambda).
		\]
	\end{proof}
	
	The index shift \(n\mapsto n+r\) and the factor
	\[
	(-1)^r\frac{M_n!}{M_{n+r}!}
	\]
	relate the Bernoulli-type interpolation values to the
	Apostol-Mersenne-Bernoulli polynomials.
	
	\section{Euler-type Mersenne-Lerch interpolation}
	
	The alternating factor \((-1)^m\) gives the Euler-type interpolation family
	and leads to the denominator \(1+\lambda e_M^t\) in its generating function.
	
	\begin{definition}\label{EulerMersLerch}
		Let \(r\in\mathbb{N}\), \(n\in\mathbb{N}_0\),
		\(x\in\mathbb{C}\), and \(|\lambda|<1\). The Euler-type
		Mersenne-Lerch interpolation value of order \(r\) at the nonpositive
		integer \(-n\) is defined by
		\[
		\Phi^{E,r}_{M,\lambda}(-n,x)
		=
		\sum_{m=0}^{\infty}
		(-1)^m
		\binom{m+r-1}{m}
		\lambda^mP_{n,M}(x;m).
		\]
	\end{definition}
	By Theorem \ref{exp1}, for fixed \(n\), the polynomial
	\(P_{n,M}(x;m)\) has degree at most \(n\) in \(m\). Hence, the defining series converges for \(|\lambda|<1\). When
	\(r=1\), we write
	\[
	\Phi^{E}_{M,\lambda}(-n,x)
	=
	\Phi^{E,1}_{M,\lambda}(-n,x)
	=
	\sum_{m=0}^{\infty}
	(-1)^m\lambda^mP_{n,M}(x;m).
	\]
	
	The generating function follows from the defining relation for
	\(P_{n,M}(x;m)\) and the negative binomial expansion.
	
	\begin{theorem}\label{eulergen}
		Let \(r\in\mathbb{N}\), \(x\in\mathbb{C}\), and
		\(|\lambda|<1\). Then
		\[
		\sum_{n=0}^{\infty}
		\Phi^{E,r}_{M,\lambda}(-n,x)\frac{t^n}{M_n!}
		=
		\frac{e_M^{xt}}{(1+\lambda e_M^t)^r},
		\]
		for all \(t\) in a neighborhood of \(0\) satisfying
		\[
		|\lambda e_M^t|<1.
		\]
	\end{theorem}
	
	\begin{proof}
		By Definition~\ref{EulerMersLerch},
		\[
		\sum_{n=0}^{\infty}
		\Phi^{E,r}_{M,\lambda}(-n,x)\frac{t^n}{M_n!}
		=
		\sum_{n=0}^{\infty}
		\left(
		\sum_{m=0}^{\infty}
		(-1)^m
		\binom{m+r-1}{m}
		\lambda^mP_{n,M}(x;m)
		\right)
		\frac{t^n}{M_n!}.
		\]
		The absolute-convergence estimate used in the proof of
		Theorem~\ref{ApostolMersBern} also applies here, so the order of
		summation may be changed. By Definition~\ref{MersenneTransPol},
		\begin{align*}
			\sum_{n=0}^{\infty}
			\Phi^{E,r}_{M,\lambda}(-n,x)\frac{t^n}{M_n!}
			&=
			\sum_{m=0}^{\infty}
			(-1)^m
			\binom{m+r-1}{m}
			\lambda^m
			\sum_{k=0}^{\infty}
			P_{k,M}(x;m)\frac{t^k}{M_k!}
			\\
			&=
			\sum_{m=0}^{\infty}
			(-1)^m
			\binom{m+r-1}{m}
			\lambda^m
			e_M^{xt}\left(e_M^t\right)^m
			\\
			&=
			e_M^{xt}
			\sum_{m=0}^{\infty}
			\binom{m+r-1}{m}
			\left(-\lambda e_M^t\right)^m.
		\end{align*}
		For \(|\lambda e_M^t|<1\), the negative binomial expansion gives
		\[
		\sum_{m=0}^{\infty}
		\binom{m+r-1}{m}
		\left(-\lambda e_M^t\right)^m
		=
		\frac{1}{(1+\lambda e_M^t)^r}.
		\]
		Hence,
		\[
		\sum_{n=0}^{\infty}
		\Phi^{E,r}_{M,\lambda}(-n,x)\frac{t^n}{M_n!}
		=
		\frac{e_M^{xt}}{(1+\lambda e_M^t)^r}.
		\]
	\end{proof}
	
	The factor \(2^r\) in the generating function of the
	Apostol-Mersenne-Euler polynomials gives the following interpolation
	formula.
	
	\begin{theorem}
		Let \(r\in\mathbb{N}\), \(n\in\mathbb{N}_0\),
		\(x\in\mathbb{C}\), and \(|\lambda|<1\). Then
		\[
		\Phi^{E,r}_{M,\lambda}(-n,x)
		=
		\frac{1}{2^r}
		E^{(r)}_{n,M}(x;\lambda),
		\]
		where \(E^{(r)}_{n,M}(x;\lambda)\) denotes the
		Apostol-Mersenne-Euler polynomial of order \(r\).
	\end{theorem}
	
	\begin{proof}
		From \eqref{AposMerEu},
		\[
		\left(
		\frac{2}{\lambda e_M^t+1}
		\right)^r
		e_M^{xt}
		=
		\sum_{n=0}^{\infty}
		E^{(r)}_{n,M}(x;\lambda)
		\frac{t^n}{M_n!}.
		\]
		Therefore,
		\begin{align}
			\frac{e_M^{xt}}{(1+\lambda e_M^t)^r}
			&=
			\frac{1}{2^r}
			\left(
			\frac{2}{\lambda e_M^t+1}
			\right)^r
			e_M^{xt}
			\nonumber\\
			&=
			\frac{1}{2^r}
			\sum_{n=0}^{\infty}
			E^{(r)}_{n,M}(x;\lambda)
			\frac{t^n}{M_n!}
			\nonumber\\
			&=
			\sum_{n=0}^{\infty}
			\frac{1}{2^r}
			E^{(r)}_{n,M}(x;\lambda)
			\frac{t^n}{M_n!}.
			\label{eq1}
		\end{align}
		On the other hand, Theorem~\ref{eulergen} gives
		\begin{equation}\label{eq2}
			\frac{e_M^{xt}}{(1+\lambda e_M^t)^r}
			=
			\sum_{n=0}^{\infty}
			\Phi^{E,r}_{M,\lambda}(-n,x)
			\frac{t^n}{M_n!}.
		\end{equation}
		Comparing the coefficients of \(\displaystyle t^n/M_n!\) in
		\eqref{eq1} and \eqref{eq2}, we obtain
		\[
		\Phi^{E,r}_{M,\lambda}(-n,x)
		=
		\frac{1}{2^r}
		E^{(r)}_{n,M}(x;\lambda).
		\]
	\end{proof}
	
	The factor \(2^{-r}\) relates the Euler-type interpolation values at
	nonpositive integers to the Apostol-Mersenne-Euler polynomials of order
	\(r\).
	
	\section{Derivative, integral, addition, and difference formulas}
	
	Suna and Ray obtained related derivative, integral, addition, and difference
	identities for Mersenne-Bernoulli, Mersenne-Euler, and Apostol-type
	Mersenne polynomial families
	\cite{SunaRay2026,SunaRay2026Mersenne}. The formulas in this section concern
	the Bernoulli-type and Euler-type interpolation values of order
	\(r\in\mathbb{N}\), with \(|\lambda|<1\), and follow directly from the
	generating functions established in the preceding sections.
		
	\begin{theorem}
		Let \(r\in\mathbb{N}\), \(x\in\mathbb{C}\), and
		\(|\lambda|<1\). For \(n\geq1\),
		\begin{equation}\label{derivative1}
			D_x\Phi^{B,r}_{M,\lambda}(-n,x)
			=
			M_n\Phi^{B,r}_{M,\lambda}(-(n-1),x),
		\end{equation}
		and
		\begin{equation}\label{derivative2}
			D_x\Phi^{E,r}_{M,\lambda}(-n,x)
			=
			M_n\Phi^{E,r}_{M,\lambda}(-(n-1),x).
		\end{equation}
		When \(n=0\),
		\begin{equation}\label{br0}
			D_x\Phi^{B,r}_{M,\lambda}(0,x)=0,
		\end{equation}
		and
		\begin{equation}\label{er0}
			D_x\Phi^{E,r}_{M,\lambda}(0,x)=0.
		\end{equation}
	\end{theorem}
	
	\begin{proof}
		By Theorem~\ref{ApostolMersBern},
		\[
		\sum_{n=0}^{\infty}
		\Phi^{B,r}_{M,\lambda}(-n,x)\frac{t^n}{M_n!}
		=
		\frac{e_M^{xt}}{(1-\lambda e_M^t)^r}.
		\]
		Applying \(D_x\) and using \(D_xe_M^{xt}=te_M^{xt}\), we obtain
		\begin{align*}
			\sum_{n=0}^{\infty}
			D_x\Phi^{B,r}_{M,\lambda}(-n,x)\frac{t^n}{M_n!}
			&=
			D_x\left(
			\frac{e_M^{xt}}{(1-\lambda e_M^t)^r}
			\right)
			\\
			&=
			t\frac{e_M^{xt}}{(1-\lambda e_M^t)^r}
			\\
			&=
			t\sum_{n=0}^{\infty}
			\Phi^{B,r}_{M,\lambda}(-n,x)\frac{t^n}{M_n!}
			\\
			&=
			\sum_{n=1}^{\infty}
			M_n\Phi^{B,r}_{M,\lambda}(-(n-1),x)
			\frac{t^n}{M_n!}.
		\end{align*}
		For \(n\geq1\), comparison of the coefficients of
		\(\displaystyle t^n/M_n!\) gives \eqref{derivative1}. The series on the
		right has no constant term, so \eqref{br0} follows.
		
		For the Euler-type interpolation values, Theorem~\ref{eulergen} gives
		\[
		\sum_{n=0}^{\infty}
		\Phi^{E,r}_{M,\lambda}(-n,x)\frac{t^n}{M_n!}
		=
		\frac{e_M^{xt}}{(1+\lambda e_M^t)^r}.
		\]
		Applying \(D_x\) gives
		\[
		\sum_{n=0}^{\infty}
		D_x\Phi^{E,r}_{M,\lambda}(-n,x)\frac{t^n}{M_n!}
		=
		\sum_{n=1}^{\infty}
		M_n\Phi^{E,r}_{M,\lambda}(-(n-1),x)
		\frac{t^n}{M_n!}.
		\]
		Comparing coefficients gives \eqref{derivative2}, while the absence of a
		constant term gives \eqref{er0}.
	\end{proof}
	
	The derivative relations give the corresponding M-integral formulas.
	
	\begin{theorem}
		Let \(r\in\mathbb{N}\), \(n\in\mathbb{N}_0\),
		\(a,b\in\mathbb{R}\), and \(|\lambda|<1\). Then
		\begin{equation}\label{integral1}
			\int_a^b
			\Phi^{B,r}_{M,\lambda}(-n,x)\,d_Mx
			=
			\frac{
				\Phi^{B,r}_{M,\lambda}(-(n+1),b)
				-
				\Phi^{B,r}_{M,\lambda}(-(n+1),a)
			}{M_{n+1}},
		\end{equation}
		and
		\begin{equation}\label{integral2}
			\int_a^b
			\Phi^{E,r}_{M,\lambda}(-n,x)\,d_Mx
			=
			\frac{
				\Phi^{E,r}_{M,\lambda}(-(n+1),b)
				-
				\Phi^{E,r}_{M,\lambda}(-(n+1),a)
			}{M_{n+1}}.
		\end{equation}
	\end{theorem}
	
	\begin{proof}
		Replacing \(n\) by \(n+1\) in \eqref{derivative1}, we have
		\[
		D_x\Phi^{B,r}_{M,\lambda}(-(n+1),x)
		=
		M_{n+1}\Phi^{B,r}_{M,\lambda}(-n,x).
		\]
		Thus,
		\[
		D_x\left(
		\frac{1}{M_{n+1}}
		\Phi^{B,r}_{M,\lambda}(-(n+1),x)
		\right)
		=
		\Phi^{B,r}_{M,\lambda}(-n,x).
		\]
		Using \eqref{Mintegral} gives \eqref{integral1}.
		
		In the same way, \eqref{derivative2} gives
		\[
		D_x\Phi^{E,r}_{M,\lambda}(-(n+1),x)
		=
		M_{n+1}\Phi^{E,r}_{M,\lambda}(-n,x),
		\]
		and \eqref{Mintegral} gives \eqref{integral2}.
	\end{proof}
	
	The M-binomial addition produces the addition formulas for the Bernoulli-type
	and Euler-type interpolation values.
	
	\begin{theorem}\label{addition}
		Let \(r\in\mathbb{N}\), \(n\in\mathbb{N}_0\),
		\(x,y\in\mathbb{C}\), and \(|\lambda|<1\). Then
		\begin{equation}\label{addition1}
			\Phi^{B,r}_{M,\lambda}(-n,x+_M y)
			=
			\sum_{k=0}^{n}
			\binom{n}{k}_M
			\Phi^{B,r}_{M,\lambda}(-k,x)y^{n-k},
		\end{equation}
		and
		\begin{equation}\label{addition2}
			\Phi^{E,r}_{M,\lambda}(-n,x+_M y)
			=
			\sum_{k=0}^{n}
			\binom{n}{k}_M
			\Phi^{E,r}_{M,\lambda}(-k,x)y^{n-k}.
		\end{equation}
	\end{theorem}
	
	\begin{proof}
		Using
		\[
		e_M^{(x+_M y)t}
		=
		e_M^{xt}e_M^{yt},
		\]
		together with Theorem~\ref{ApostolMersBern} and
		\eqref{Mexponentialcoeffi}, we obtain
		\begin{align*}
			\sum_{n=0}^{\infty}
			\Phi^{B,r}_{M,\lambda}(-n,x+_M y)
			\frac{t^n}{M_n!}
			&=
			\frac{e_M^{(x+_M y)t}}
			{(1-\lambda e_M^t)^r}
			\\
			&=
			\frac{e_M^{xt}}
			{(1-\lambda e_M^t)^r}
			e_M^{yt}
			\\
			&=
			\left(
			\sum_{k=0}^{\infty}
			\Phi^{B,r}_{M,\lambda}(-k,x)
			\frac{t^k}{M_k!}
			\right)
			\left(
			\sum_{j=0}^{\infty}
			y^j\frac{t^j}{M_j!}
			\right)
			\\
			&=
			\sum_{n=0}^{\infty}
			\left(
			\sum_{k=0}^{n}
			\frac{
				\Phi^{B,r}_{M,\lambda}(-k,x)y^{n-k}
			}{
				M_k!M_{n-k}!
			}
			\right)t^n
			\\
			&=
			\sum_{n=0}^{\infty}
			\left(
			\sum_{k=0}^{n}
			\binom{n}{k}_M
			\Phi^{B,r}_{M,\lambda}(-k,x)y^{n-k}
			\right)
			\frac{t^n}{M_n!}.
		\end{align*}
		Comparing coefficients gives \eqref{addition1}.
		
		For the Euler-type interpolation values, Theorem~\ref{eulergen} gives
		\[
		\sum_{n=0}^{\infty}
		\Phi^{E,r}_{M,\lambda}(-n,x+_M y)
		\frac{t^n}{M_n!}
		=
		\frac{e_M^{xt}}
		{(1+\lambda e_M^t)^r}
		e_M^{yt}.
		\]
		The same Cauchy product gives \eqref{addition2}.
	\end{proof}
	
	For the difference formulas, define the order-zero values by
	\begin{equation}\label{convention}
		\Phi^{B,0}_{M,\lambda}(-n,x)
		=
		\Phi^{E,0}_{M,\lambda}(-n,x)
		=
		x^n,
		\qquad n\in\mathbb{N}_0.
	\end{equation}
	
	\begin{theorem}\label{difference}
		Let \(r\in\mathbb{N}\), \(n\in\mathbb{N}_0\),
		\(x\in\mathbb{C}\), and \(|\lambda|<1\). Then
		\begin{equation}\label{difference1}
			\lambda\Phi^{B,r}_{M,\lambda}(-n,x+_M1)
			-
			\Phi^{B,r}_{M,\lambda}(-n,x)
			=
			-\Phi^{B,r-1}_{M,\lambda}(-n,x),
		\end{equation}
		and
		\begin{equation}\label{difference2}
			\lambda\Phi^{E,r}_{M,\lambda}(-n,x+_M1)
			+
			\Phi^{E,r}_{M,\lambda}(-n,x)
			=
			\Phi^{E,r-1}_{M,\lambda}(-n,x).
		\end{equation}
		In particular, when \(r=1\),
		\begin{equation}\label{br1}
			\lambda\Phi^{B}_{M,\lambda}(-n,x+_M1)
			-
			\Phi^{B}_{M,\lambda}(-n,x)
			=
			-x^n,
		\end{equation}
		and
		\begin{equation}\label{er11}
			\lambda\Phi^{E}_{M,\lambda}(-n,x+_M1)
			+
			\Phi^{E}_{M,\lambda}(-n,x)
			=
			x^n.
		\end{equation}
	\end{theorem}
	
	\begin{proof}
		By Theorem~\ref{ApostolMersBern} and
		\[
		e_M^{(x+_M1)t}
		=
		e_M^{xt}e_M^t,
		\]
		we have
		\[
		\sum_{n=0}^{\infty}
		\Phi^{B,r}_{M,\lambda}(-n,x+_M1)
		\frac{t^n}{M_n!}
		=
		\frac{e_M^{xt}e_M^t}
		{(1-\lambda e_M^t)^r}.
		\]
		Hence,
		\begin{align*}
			&\sum_{n=0}^{\infty}
			\left[
			\lambda\Phi^{B,r}_{M,\lambda}(-n,x+_M1)
			-
			\Phi^{B,r}_{M,\lambda}(-n,x)
			\right]
			\frac{t^n}{M_n!}
			\\
			&\qquad=
			\lambda
			\frac{e_M^{xt}e_M^t}
			{(1-\lambda e_M^t)^r}
			-
			\frac{e_M^{xt}}
			{(1-\lambda e_M^t)^r}
			\\
			&\qquad=
			\frac{e_M^{xt}(\lambda e_M^t-1)}
			{(1-\lambda e_M^t)^r}
			\\
			&\qquad=
			-\frac{e_M^{xt}}
			{(1-\lambda e_M^t)^{r-1}}.
		\end{align*}
		For \(r\geq2\), Theorem~\ref{ApostolMersBern} gives
		\[
		-\frac{e_M^{xt}}
		{(1-\lambda e_M^t)^{r-1}}
		=
		-\sum_{n=0}^{\infty}
		\Phi^{B,r-1}_{M,\lambda}(-n,x)
		\frac{t^n}{M_n!}.
		\]
		When \(r=1\), the expression reduces to
		\[
		-e_M^{xt}
		=
		-\sum_{n=0}^{\infty}
		x^n\frac{t^n}{M_n!},
		\]
		which agrees with \eqref{convention}. Comparing coefficients gives
		\eqref{difference1} and \eqref{br1}.
		
		For the Euler-type interpolation values, Theorem~\ref{eulergen} gives
		\[
		\sum_{n=0}^{\infty}
		\Phi^{E,r}_{M,\lambda}(-n,x+_M1)
		\frac{t^n}{M_n!}
		=
		\frac{e_M^{xt}e_M^t}
		{(1+\lambda e_M^t)^r}.
		\]
		Therefore,
		\begin{align*}
			&\sum_{n=0}^{\infty}
			\left[
			\lambda\Phi^{E,r}_{M,\lambda}(-n,x+_M1)
			+
			\Phi^{E,r}_{M,\lambda}(-n,x)
			\right]
			\frac{t^n}{M_n!}
			\\
			&\qquad=
			\lambda
			\frac{e_M^{xt}e_M^t}
			{(1+\lambda e_M^t)^r}
			+
			\frac{e_M^{xt}}
			{(1+\lambda e_M^t)^r}
			\\
			&\qquad=
			\frac{e_M^{xt}(\lambda e_M^t+1)}
			{(1+\lambda e_M^t)^r}
			\\
			&\qquad=
			\frac{e_M^{xt}}
			{(1+\lambda e_M^t)^{r-1}}.
		\end{align*}
		For \(r\geq2\),
		\[
		\frac{e_M^{xt}}
		{(1+\lambda e_M^t)^{r-1}}
		=
		\sum_{n=0}^{\infty}
		\Phi^{E,r-1}_{M,\lambda}(-n,x)
		\frac{t^n}{M_n!}.
		\]
		When \(r=1\),
		\[
		e_M^{xt}
		=
		\sum_{n=0}^{\infty}
		x^n\frac{t^n}{M_n!}.
		\]
		Comparing coefficients gives \eqref{difference2} and \eqref{er11}.
	\end{proof}
	
	\section{Comparison with the \(q=2\) case}
	
	The Mersenne quantities used in this paper arise from the usual
	\(q\)-calculus expressions when \(q=2\). The \(q\)-integer,
	\(q\)-factorial, and \(q\)-binomial coefficient are defined by
	\cite{KacCheung2002}
	\[
	[n]_q
	=
	\frac{q^n-1}{q-1},
	\]
	\[
	[n]_q!
	=
	[n]_q[n-1]_q\cdots[1]_q,
	\qquad
	[0]_q!=1,
	\]
	and
	\[
	\binom{n}{k}_q
	=
	\frac{[n]_q!}{[k]_q![n-k]_q!}.
	\]
	Setting \(q=2\) gives
	\[
	[n]_2
	=
	2^n-1
	=
	M_n.
	\]
	Consequently,
	\[
	[n]_2!=M_n!,
	\]
	and
	\begin{align*}
		\left.
		\binom{n}{k}_q
		\right|_{q=2}
		&=
		\frac{[n]_2!}{[k]_2![n-k]_2!}
		\\
		&=
		\frac{M_n!}{M_k!M_{n-k}!}
		\\
		&=
		\binom{n}{k}_M.
	\end{align*}
	
	The M-derivative is obtained from the Jackson \(q\)-derivative in the same
	way. For \(q\neq1\), the Jackson derivative is
	\[
	D_{q,x}f(x)
	=
	\frac{f(qx)-f(x)}{(q-1)x}.
	\]
	Its value at \(q=2\) is
	\[
	D_{2,x}f(x)
	=
	\frac{f(2x)-f(x)}{x}
	=
	D_xf(x).
	\]
	For the monomial \(x^n\),
	\[
	D_{2,x}(x^n)
	=
	[n]_2x^{n-1}
	=
	M_nx^{n-1}
	=
	D_x(x^n).
	\]
	
	Consider the factorial-based \(q\)-exponential
	\[
	e_q(z)
	=
	\sum_{n=0}^{\infty}
	\frac{z^n}{[n]_q!}.
	\]
	Its specialization at \(q=2\) satisfies
	\begin{align*}
		e_2(xt)
		&=
		\sum_{n=0}^{\infty}
		\frac{x^nt^n}{[n]_2!}
		\\
		&=
		\sum_{n=0}^{\infty}
		\frac{x^nt^n}{M_n!}
		\\
		&=
		e_M^{xt}.
	\end{align*}
	
	Thus, the M-factorial, M-binomial coefficient, M-derivative, and
	M-exponential are the \(q=2\) forms of their corresponding \(q\)-calculus
	expressions. The Mersenne notation is retained because the two interpolation
	families are associated with the Apostol-Mersenne polynomial families. In
	this notation, the special values are
	\[
	\Phi^{B,r}_{M,\lambda}(-n,x)
	=
	(-1)^r
	\frac{M_n!}{M_{n+r}!}
	B^{(r)}_{n+r,M}(x;\lambda),
	\]
	and
	\[
	\Phi^{E,r}_{M,\lambda}(-n,x)
	=
	\frac{1}{2^r}
	E^{(r)}_{n,M}(x;\lambda).
	\]
	These are the \(q=2\) interpolation relations for the
	Apostol-Mersenne-Bernoulli and Apostol-Mersenne-Euler polynomials.

\section{Conclusion}

Bernoulli-type and Euler-type Mersenne-Lerch interpolation families were
constructed from the Mersenne translation polynomials \(P_{n,M}(x;m)\),
defined by
\[
e_M^{xt}\left(e_M^t\right)^m
=
\sum_{n=0}^{\infty}
P_{n,M}(x;m)\frac{t^n}{M_n!}.
\]
These polynomials have an explicit M-multinomial formula and an expansion in
terms of M-Stirling numbers of the second kind.

The generating functions of the two interpolation families yield
\[
\Phi^{B,r}_{M,\lambda}(-n,x)
=
(-1)^r
\frac{M_n!}{M_{n+r}!}
B^{(r)}_{n+r,M}(x;\lambda),
\]
and
\[
\Phi^{E,r}_{M,\lambda}(-n,x)
=
\frac{1}{2^r}
E^{(r)}_{n,M}(x;\lambda).
\]
Hence, the Apostol-Mersenne-Bernoulli and
Apostol-Mersenne-Euler polynomials are obtained from the interpolation
values at nonpositive integers. The same generating functions give the
derivative, integral, addition, and difference formulas for both families.

The comparison with \(q\)-calculus shows that the M-factorial,
M-binomial coefficient, M-derivative, and M-exponential are the \(q=2\)
forms of the corresponding \(q\)-quantities. The Mersenne notation keeps
the interpolation formulas connected with the Apostol-Mersenne polynomial
families.

A remaining question is whether functions of a complex variable \(s\) can be
constructed to interpolate these values, and how such functions relate to
Mersenne-type Hurwitz-Lerch factorial zeta functions involving \(M_n!\).

\end{document}